\documentclass{amsart}

\usepackage{amsmath,amssymb,graphicx,url} 
\usepackage{tikz}
\usetikzlibrary{arrows,backgrounds,decorations,automata}

\newtheorem{thm}{Theorem}[section]
\newtheorem{lem}[thm]{Lemma}
\newtheorem{cor}[thm]{Corollary}
\newtheorem{prop}[thm]{Proposition}
\theoremstyle{definition}

\newcommand{\bi}{\begin{itemize}}
\newcommand{\ei}{\end{itemize}}
\newcommand{\be}{\begin{enumerate}}
\newcommand{\ee}{\end{enumerate}}
\newcommand{\bc}{\begin{center}}
\newcommand{\ec}{\end{center}}
\newcommand{\bt}{\begin{tabular}}
\newcommand{\et}{\end{tabular}}
\newcommand{\ba}{\begin{array}}
\newcommand{\ea}{\end{array}}

\newcommand{\N}{\mathbb N}

\newcommand{\ff}{F}
\newcommand{\quot}{/\hspace{-1mm}/}

\def\ident{e}
\def\lcm{\operatorname{lcm}}

\newcommand{\GalJanHNN}{MR2003149}

\newcommand{\Bass}{MR1239551}
\newcommand{\Serre}{MR1954121}

\newcommand{\ShalomW}{ShalomW}

\newcommand{\Patty}{MR1199571}
\newcommand{\WillStructure}{MR1299067}
\newcommand{\WillTidy}{MR2052362}
\newcommand{\Moller}{MR1913920}
\newcommand{\HewittRoss}{MR551496}

\newcommand{\Tzanev}{MR2015025}
\newcommand{\Schlichting}{MR583752}
\newcommand{\GlockWillis}{MR2032838}
\newcommand{\ElderA}{MR2776987}
\newcommand{\ElderB}{MR2142503}
\newcommand{\ElderC}{MR2126733}
\newcommand{\ElderEEO}{EEO}
\newcommand{\WillFurther}{MR1813900}
\newcommand{\LS}{MR0577064}
\newcommand{\vanD}{MR1556954}
\newcommand{\ColinProfinite}{ColinProfinite}
\newcommand{\BurgerMozes}{MR1839488}
\newcommand{\RibesZ}{MR2599132}
\newcommand{\BaWi}{MR2085717}
\newcommand{\HG}{MR2243237}

\begin{document}

\title[Totally disconnected  Baumslag-Solitar groups]{Totally disconnected groups from Baumslag-Solitar groups}
\author{Murray Elder}
\address{School of Mathematical and Physical Sciences, 
The University of Newcastle,
Callaghan NSW 2308, Australia}
\email{Murray.Elder@newcastle.edu.au}
\author{George Willis}
\address{School of Mathematical and Physical Sciences, 
The University of Newcastle,
Callaghan NSW 2308, Australia}
\email{George.Willis@newcastle.edu.au}
\date{\today}
\thanks{
Research supported by the Australian Research Council projects  DP0984342, DP120100996  and FT110100178}

\begin{abstract}
For each Baumslag-Solitar group $\mathrm{BS}(m,n)$ ($m,n\in \mathbb{Z}\setminus\{0\}$), a totally disconnected, locally compact group, $G_{m,n}$,
is constructed so that $\mathrm{BS}(m,n)$ is identified with a dense subgroup of $G_{m,n}$. The scale function on $G_{m,n}$, a structural invariant for the topological group, is seen to distinguish the parameters $m$ and $n$ to the extent that the set of scale values is 
$$
\left\{  \left(\frac{\lcm(m,n)}{|m|}\right)^{\rho}, \left(\frac{\lcm(m,n)}{|n|}\right)^{\rho} \ \vert \ \rho\in \N\right\}.
$$
It is also shown that $G_{m,n}$ has flat rank 1 when $|m|\neq  |n|$ and 0 otherwise, and that $G_{m,n}$ has a compact, open subgroup isomorphic to the product  $\left\{(\mathbb{Z}_p,+) \mid p\hbox{ is a prime divisor of the scale}\right\}$.
\end{abstract}

\maketitle

\section{Introduction}
\label{sec:intro}

In this article we investigate a family of totally disconnected,   
locally compact groups $G_{m,n}$ built from  Baumslag-Solitar groups, which are the groups  $\mathrm{BS}(m,n)$ with presentations
$\displaystyle  \left\langle a,t \ | \ ta^mt^{-1}=a^n\right\rangle$
for  non-zero integers $m$ and $n$.

Recall that a topological group $G$ is a group that is also a topological space such that  the map $(x,y)\mapsto xy^{-1}:G\times G\rightarrow G$ 
 is continuous. 
 A topological space is {\em Hausdorff} if for any two points $x,y$ there are disjoint open sets ${U},{V}$, with $x\in {U}$ and $y\in {V}$,
  {\em  totally disconnected} if the connected component containing any point is just the point itself, or equivalently that for any two points $x,y$ there are disjoint open sets ${U},{V}$, with $x\in {U}$ and $y\in {V}$ and whose union is the whole space, and  {\em locally compact} if each point has a compact neighbourhood. The study of locally compact groups naturally splits into the study of connected and totally disconnected groups, since for every locally compact group $G$, the connected component $G_0$ of the identity is a closed normal subgroup, and $G/G_0$ is totally disconnected.

The
 subgroup $\left\langle a\right\rangle$ is 
{\em commensurated} by $\mathrm{BS}(m,n)$, and we follow a construction described in \cite{\GlockWillis, \ColinProfinite, \Schlichting, \ShalomW, \Tzanev}  that uses a commensurated subgroup of an abstract group to construct  a totally disconnected, locally compact group 
that is the completion of the abstract group with respect to a topology defined on it in terms of cosets of the commensurated subgroup.
The group thus obtained from $\mathrm{BS}(m,n)$ is denoted $G_{m,n}$.

In \cite{\GalJanHNN} Gal and Januszkiewicz prove that Baumslag-Solitar groups are a-T-menable, by embedding them into the  topological group of automorphisms of their corresponding {\em Bass-Serre} tree. The closure of $\mathrm{BS}(m,n)$ in this group coincides with the group $G_{m,n}$ we construct.

A key invariant of a totally disconnected locally compact group is the set of positive integers corresponding to the {\em scales} of its elements, which we define in Section \ref{sec:scale}.
The scale function and the related notion  of minimizing subgroups were introduced by the second author \cite{\WillStructure, \WillFurther} and fill a similar
role in the theory of totally disconnected groups to that played by
eigenvalues and triangularising bases in linear algebra and the
theory of matrix groups.

 We prove that the set of scales for the group $G_{m,n}$ for $m,n\neq 0$  is $$\displaystyle \left\{  \left(\frac{\lcm(m,n)}{|m|}\right)^{\rho}, \left(\frac{\lcm(m,n)}{|n|}\right)^{\rho} \ \vert \ \rho\in \N\right\}.$$ 
The scale thus partially distinguishes the parameters in the definition of $\mathrm{BS}(m,n)$. In particular, for each pair of relatively prime integers $m,n>0$ we obtain a distinct nondiscrete totally disconnected group. 
We offer two alternative proofs of this result: a combinatorial approach is given in Section \ref{sec:compute}; while in Section \ref{sec:local} the {\em tidying procedure}  developed by the second author is used.

Another key invariant of totally disconnected locally compact groups  is their {\em flat rank} (defined in Section \ref{sec:flatRank}). 
Here we show that  $G_{m,n}$ has flat rank 1 for all $|m|\neq |n|$.



The article is organised as follows.  In Sections \ref{sec:scale}--\ref{sec:comm} we define scale, Baumslag-Solitar groups and commensurated subgroups, and list   properties that will be used later. In Section \ref{sec:construction} we outline the construction of a totally disconnected group starting with an abstract group having a commensurated subgroup, which we apply to Baumslag-Solitar groups.
In Section \ref{sec:compute} we compute the scale function for $G_{m,n}$. In Section \ref{sec:flatRank} we compute a formula for the {\em modular function} for $G_{m,n}$,  and compute  its  flat rank. In Section  \ref{sec:local} we give more detail on the local structure of $G_{m,n}$, describing its compact open subgroups explicitly, and recompute the modular and scale functions  with this description. 

Note that throughout this paper $\N$ will denote the set of nonnegative integers (including 0).

 The authors wish to thank  Mathieu Carette and Yves de Cornulier for suggestions and corrections to earlier drafts.

\section{The scale function}
  \label{sec:scale}
  Let $G$ be a totally disconnected locally compact group. By 
 van Dantzig's Theorem (see \cite{\vanD} or \cite{\HewittRoss} Theorem 7.7),   every neighourhood of the identity contains a compact open subgroup, $V$ say.  An automorphism of $G$ is a group automorphism $\alpha:G\rightarrow G$  that is also a topological homeomorphism, meaning $\alpha$ and $\alpha^{-1}$ are continuous.
If $\alpha$ is an automorphism and $V$ is a compact open subgroup, 
 the set  $\alpha^{-1}(V)$ is compact 
 and open
 .  The cosets of the subgroup 
 $\alpha^{-1}({V})\cap {V}$  form an open cover of  ${V}$, and since ${V}$ is compact, the index $[V:  \alpha^{-1}(V)\cap V]$ is finite. 
 Define the \emph{scale} of $\alpha$, denoted $s(\alpha)$, to be the minimum such
index over all compact open subgroups ${V}$ of $G$. Then $s(\alpha)$
is the minimum of a set of positive integers and $s:\mathrm{Aut}(G)\rightarrow \mathbb Z^+$ is a well defined function.  A
subgroup ${V}$ for which $s(\alpha)$ is attained is called {\em minimizing} for
$\alpha$. 

In the case that $\alpha=\alpha_x:g\mapsto xgx^{-1}$ is an inner automorphism,
the scale function induces a function from the group to 
$\mathbb Z^+$, also denoted by $s$, which enjoys the 
 the following properties.
\begin{prop}[\cite{\WillStructure,\WillFurther}]\label{prop:scale_properties}
Let $s:G\rightarrow \mathbb Z^+$ be the scale function on $G$. Then 
\begin{enumerate}
\item[$(i)$] $s$ is continuous;
\item[$(ii)$]    for each $x\in G$ and $n\in\N$, 
$s(x^n)=s(x)^n$;
\item[$(iii)$]  if ${V}$ is minimising for $x$ then ${V}$ is minimising for $x^i$ for all $i\in \mathbb Z$;
\item[$(iv)$]  
$s$  is invariant under conjugation, that is,  $s(x)=s(yxy^{-1})$ for any $x,y\in G$;
\item[$(v)$] $s(x)=s(x^{-1})=1$ if and only if there is a compact open subgroup $V$ with $x^{-1}Vx=V$.  

\end{enumerate}
\end{prop}
Note that if $G$ has the discrete topology then the scale of any element is 1, since ${V}=\{1\}$ is  open and compact (and obviously for each $x\in G, x^{-1}\{1\}x=\{1\}$).

The continuity of the scale function implies the following.
\begin{cor}\label{cor:dense}
If $B$ is a dense subset of a totally disconnected locally compact group $G$, then $\{s(b)\ | \ b\in B\}= \{s(g)\ | \ g\in G\}$.
\end{cor}
\begin{proof}
If $g\in G$ has scale $s(g)=n$, then the inverse image $U$ of the open set $\{n\}$ in $\mathbb Z^+$ under the scale function is an open set in $G$. 
Since $B$ is dense in $G$, every open set contains points from $B$, so there is a point $b\in U\cap B$ with $s(b)=n$.
\end{proof}

We
 will make use of an asymptotic formula of M\"oller which makes it possible to use arbitrary compact open subgroups  to calculate the scale function.
\begin{thm}[\cite{\Moller} Theorem 7.7]\label{thm:Moller}
For any compact open subgroup ${V}$,
\begin{equation*}
s(x)=\lim_{k\rightarrow \infty} \left[{V} : {V}\cap
x^{-k}{V}x^k\right]^{\frac1{k}}.\end{equation*} \end{thm}

\section{Baumslag-Solitar groups}
  \label{sec:bsintro}

  In this section we collect some facts that will be useful in computing scales later on. Let $\mathrm{BS}(m,n)$ be the group with presentation
$\displaystyle  \left\langle a,t \ | \ ta^mt^{-1}=a^n\right\rangle$
for  non-zero integers $m$ and $n$. 

Define $\rho(w)$ to be the  number of  $t$ letters minus the number of $t^{-1}$ letters in the word $w\in\{a^{\pm 1}, t^{\pm 1}\}^*$.   Recall  Britton's lemma \cite{\LS}, which states that if a freely reduced nonempty word in the generators $a^{\pm 1}, t^{\pm 1}$ for $\mathrm{BS}(m,n)$ is equal to the identity element, it must contain a subword of the form $ta^{cm}t^{-1}$ or $t^{-1}a^{cn}t$ for some $c\in \mathbb Z$ (such a  subword is called a {\em pinch}). Replacing $ta^{cm}t^{-1}$ or $t^{-1}a^{cn}t$ by $a^{cn}$ or $a^{cm}$ in a word is called {\em removing a pinch}.

\begin{lem}\label{lem:texpinvariant}
 If $w,u\in \{a^{\pm 1}, t^{\pm 1}\}^*$ represent the same element in $\mathrm{BS}(m,n)$ then $\rho(w)=\rho(u)$.
  \end{lem}
\begin{proof}
If $w$ or $u$ contains a pinch,  removing them does not change the respective values of $\rho$, so remove them (since each word has a finite number of $t$ letters, this process will terminate).
Since $wu^{-1}$ equals the identity in $\mathrm{BS}(m,n)$ it must contain a pinch by Britton's lemma.
Assuming there are no pinches in $w$ or $u^{-1}$, the pinch 
 must consist of one $t$ letter in one subword and one $t^{-1}$ letter in the other. Removing all pinches in $wu^{-1}$ until we obtain the empty word gives the result.
\end{proof}

It follows that $\rho$ is an invariant of group elements, and is called the {\em $t$-exponent sum} for $w$. The first author has made extensive use of the $t$-exponent sum to prove facts about Baumslag-Solitar groups \cite{\ElderEEO,\ElderB,\ElderA,\ElderC}.

\begin{lem}\label{lem:normalformBS1n}
Each element in $\mathrm{BS}(1,n)$ can be represented uniquely in the form $t^{-p}a^qt^r$ with $p,q,r\in\mathbb Z$,  $p,r\geq 0$ and $n$ dividing $ q$ only if $p=0$ or $r=0$.
\end{lem}
\begin{proof}
Let $w\in \{a^{\pm 1},t^{\pm 1}\}^*$ be a word representing an element of $\mathrm{BS}(1,n)$. Applying the moves $ta^{\pm 1}\rightarrow a^{\pm n}t$, $a^{\pm 1}t^{-1} \rightarrow t^{-1}a^{\pm n}$, and free cancellation, $w$ is equal in the group to a word of the form $t^{-p}a^qt^r$ with $p,q,r\in\mathbb Z$,  $p,r\geq 0$. If $n$ divides $q$ and $p,r>0$ then the word contains a pinch $t^{-1}a^{nc}t$ which can be replaced by $a^c$. Repeating this gives the result.
\end{proof}

\noindent Note that $\mathrm{BS}(1,n)$ is isomorphic to the linear group, generated by $$ a= \left[\begin{array}{cc} 1 & 1\\ 0 & 1\end{array}\right] \ \ \ \mathrm{and} \ \ \ \displaystyle t=\left[\begin{array}{cc} n & 0\\ 0 & 1\end{array}\right].$$
This group is isomorphic to
\begin{equation}
\label{eq:matrix_group}
G = \left\{ \left[\begin{array}{cc} n^\rho & z\\ 0 & 1\end{array}\right] \mid  \rho\in\mathbb{Z},\ z\in \mathbb{Z}[1/n]
\right\}
\end{equation}
and Lemma~\ref{lem:normalformBS1n} thus represents a matrix in this  group as
$$
\left[\begin{array}{cc} n^{-p} & 0\\ 0 & 1\end{array}\right]\left[\begin{array}{cc} 1 & q\\ 0 & 1\end{array}\right]\left[\begin{array}{cc} n^r & 0\\ 0 & 1\end{array}\right].
$$
 

  \begin{lem}
\label{lem:BSnormal}
The subgroup $\left\langle a\right\rangle$ has a nontrivial subgroup that is normal in $\mathrm{BS}(m,n)$ if and only if $|m|=|n|$.
\end{lem}
\begin{proof}
If $|m|= |n|$, then $\left\langle a^m\right\rangle$ is normal, in fact central, in $\mathrm{BS}(m,n)$.
If $|m|\neq |n|$, assume without loss of generality that $|m|<|n|$, and suppose $K\leq \left\langle a \right\rangle$ is a nontrivial normal subgroup of $\mathrm{BS}(m,n)$. Then $K=\left\langle a^k\right\rangle$ for some positive integer $k$.
Let $s$ be the largest integer such that $k=q|n|^s+r$ for some $0\leq q<|n|$ and $0\leq r<|n|^s$. If $r=0$ then  $t^{-s}a^kt^s=a^{q(\pm m)^s}$ which is not in $K$ since $|qm^s|<|qn^s|=k$, and if $r>0$ then   $t^{-s}a^kt^s$ is not in $\left\langle a\right\rangle$.
\end{proof}

\section{Commensurated subgroups}\label{sec:comm}

  Define a relation $\sim$ on the set of subgroups of an abstract group $G$ by $H\sim K$ if $H\cap K$ is finite index in both $H$ and $K$. We say that $H$ and $K$
are {\em commensurable} if $H\sim K$. One may verify that  being 
 commensurable is an equivalence relation\footnote{
Clearly $\sim$ is reflexive and symmetric.  That $\sim$ is transitive follows because, for subgroups $H$, $K$ and $L$,  $[H:H\cap L]\leq [H:H\cap K\cap L] = [H:H\cap K][H\cap K: H\cap K\cap L]$  and $[H\cap K: H\cap K\cap L]\leq [K:K\cap L]$.}.
 
 A subgroup $H$ is {\em commensurated} by $G$ if for each $g\in G$ the subgroups $H$ and $gHg^{-1}$ are commensurable. 
 If $x,y\in G$, $xHx^{-1}\sim H$ and $H \sim yHy^{-1}$  then $xyHy^{-1}x^{-1}\sim H$ (since $H\sim K$ implies $xHx^{-1}\sim xKx^{-1}$).
It follows that if $G$ is generated by a set $X$, then $H$ is commensurated by $G$ if and only if $xHx^{-1}\sim H$ for all $x\in X$. This gives a fast way to check for commensurated subgroups in finitely generated groups.

One example of a commensurated subgroup is 
 $\mathrm{SL}(k,\mathbb Z)$ in $\mathrm{SL}(k,\mathbb Q)$, and in this important case the construction about to be given in Section~\ref{sec:construction} yields the embedding of $\mathrm{SL}(k,\mathbb Q)$ into $\mathrm{SL}(k,{\mathbb A}_f)$, where ${\mathbb A}_f$ denotes the ring of finite adeles. Restricting to the subgroup $G$ of $\mathrm{SL}(2,\mathbb Q)$ appearing in Equation~(\ref{eq:matrix_group}), this construction yields an embedding of $G$ into a group isomorphic to 
 \begin{equation}
\label{eq:matrix_group2}
\left\{ \left[\begin{array}{cc} n^\rho & z\\ 0 & 1\end{array}\right] \mid  \rho\in\mathbb{Z},\ z\in \mathbb{Q}_{p_1}\times \cdots \times \mathbb{Q}_{p_l} 
\right\},
\end{equation}
where $p_1$,\dots, $p_l$ are the prime divisors of $n$. 
In the case of  Baumslag-Solitar groups, the cyclic subgroup $\langle a \rangle$ is commensurated because $t\langle a\rangle t^{- 1} \cap \langle a\rangle=\langle a^n \rangle$ is finite index in both $t\langle a\rangle t^{-1}$ and $\langle a\rangle$, and similarly $t^{-1}\langle a\rangle t \cap \langle a\rangle=\langle a^m \rangle$   is finite index in both $t^{-1}\langle a\rangle t$ and $\langle a\rangle$. Recalling the isomorphism between $\mathrm{BS}(1,n)$ and $G$, applying the construction to $\mathrm{BS}(m,n)$ will thus extend this embedding of matrix groups.

\section{Construction of the totally disconnected group}
\label{sec:construction}

Let $G$ be an abstract group with commensurated subgroup $H$.
The action of $G$ on $G/H$ given by $g'.gH = (g'g)H$, $(g,g'\in G)$, determines a homomorphism $\pi : G \to \mathrm{Sym}(G/H)$. 
For each $x\in \mathrm{Sym}(G/H)$ and each finite subset $\ff$ of $G/H$, define 
$$
\mathcal N(x,\ff)=\{y\in \mathrm{Sym}(G/H) \ | \ y(gH)=x(gH) \ \forall (gH)\in \ff\}.
$$ 
Note that if $y\in  \mathcal N(x,\ff)$ then  $\mathcal N(x,\ff)= \mathcal N(y,\ff)$.
If  $
\mathcal N(x_1, \ff_1)\cap \mathcal N(x_2,\ff_2)  $ is nonempty then it contains some element $y$, so $N(x_1, \ff_1)=N(y,\ff_1)$ and  $N(x_2, \ff_2)=N(y,\ff_2)$. Then   $$\mathcal N(x_1, \ff_1)\cap \mathcal N(x_2,\ff_2)=
\mathcal N(y, \ff_1)\cap \mathcal N(y,\ff_2)  =\mathcal N(y,\ff_1\cup \ff_2).$$
It follows that $\left\{ \mathcal N(x,\ff)\mid x\in G,\, \ff\subseteq G/H\hbox{ finite}\right\}$ forms a basis for a topology on  $\mathrm{Sym}(G/H)$.

\begin{lem} The topology defined on $\mathrm{Sym}(G/H)$ is Hausdorff. Hence the subspace topology on $\pi(G)$ induces a Hausdorff topology on $G/\ker\pi$.
\end{lem}
\begin{proof}
If $x,y\in \mathrm{Sym}(G/H)$ are distinct, then the neighbourhoods $\mathcal N(x,\{gH\})$ and $\mathcal N(y,\{gH\})$ are disjoint for some $g\in G$. Hence $\mathrm{Sym}(G/H)$ is a Hausdorff topological group. The second claim is justified by the First Isomorphism Theorem.
\end{proof}


Note that $\ker\pi$ is a subgroup of $H$ and that $H/\ker\pi$ is commensurated by $G/\ker\pi$. The topology on $G/\ker\pi$ may then be defined equivalently by considering the injection of $G/\ker\pi$ into $\mathrm{Sym}(G/\ker\pi)/(H/\ker\pi)$. From now on it is assumed that the kernel is trivial.

\begin{lem}
The topology defined is totally disconnected.\end{lem}
\begin{proof}
If $x,y\in \mathrm{Sym}(G/H)$ are distinct, there is a coset $gH$ with $x(gH)\neq y(gH)$. Then $\mathcal{N}(x,\{gH\})$ is an open set containing $x$, and its complement, 
$$
\bigcup\left\{ \mathcal{N}(z,\{gH\}) \mid z(gH) \ne x(gH) \right\},
$$ 
is open and contains $y$.
\end{proof}


It is a standard result that 
$\mathrm{Sym}(G/H)$ is a topological group with the topology defined. For 
completeness we include a proof of this fact.
\begin{lem}
The map $\ast:(x,y)\mapsto xy^{-1}$ is continuous.
\end{lem}
\begin{proof}
The map is continuous 
at $(x,y)$ if for any open set ${V}$ containing $xy^{-1}$
there is 
 an open set ${U}\subseteq \mathrm{Sym}(G/H)\times \mathrm{Sym}(G/H)$ containing $(x,y)$ whose image is in $V$.

Since ${V}$ is open and contains $xy^{-1}$, it contains a set $\mathcal N(xy^{-1},\ff)$ for some finite $\ff$.
Take ${U}
=\mathcal N(x,y^{-1}\ff)\times \mathcal N(y,y^{-1}\ff)$. If 
 $(a,b)\in {U}$, then $b(y^{-1}(gH))=y(y^{-1}(gH))=(gH)$ and $b^{-1}(gH)=y^{-1}(gH)$ for all $gH\in \ff$ and $
ab^{-1}(gH)=a(y^{-1}(gH))
=xy^{-1}(gH)
$
for all $gH\in \ff$. Hence $ab^{-1}\in \mathcal N(xy,\ff)\subseteq {V}$.
\end{proof}


It follows that $\pi(G)$ is also a (Hausdorff and totally disconnected) topological group, with the subspace topology induced from $\mathrm{Sym}(G/H)$. 
Define $G\quot H$ to be the closure of $\pi(G)$ in $\mathrm{Sym}(G/H)$, and define $\tilde{H}$ to be the closure of $\pi(H)$. Then 
$$
\tilde{H} = \left\{ x\in G\quot H\mid x(H) = H\right\} = G\quot H\cap \mathcal N(\ident, \{H\})
$$
is an open subgroup of $G\quot H$.

It is shown next that $G\quot H$ is locally compact. The following fact is used in this argument and also later in the calculation of the scale. 
\begin{lem}
\label{lem:index_identity}
Suppose that ${V}$ is an open subgroup of $G\quot H$.
Then, setting ${U} = {V}\cap \pi(G)$, each $U$-orbit $U.(gH)\subseteq G/H$ is equal to $V$-orbit $V.(gH)$.
\end{lem}
\begin{proof}
Since ${V}$ is open and $\pi(G)$ is dense in $G\quot H$, ${U}$ is a dense subgroup of ${V}$. Hence the intersection $\mathcal N(x, \{gH\})\cap {U}$ is nonempty for any 
 $x\in {V}$ and there is $u\in U$ such that $x(gH)=u(gH)$.
 \end{proof}
 



When applied to the orbits $U.xH$ and $V.xH$, the Orbit-Stabiliser theorem implies the following consequence. 

\begin{cor}\label{cor:indexG}
If  ${V}$ and  ${U}$ are as in Lemma \ref{lem:index_identity} and $x\in G\quot H$, then $$[ {V}: x^{-1} Vx\cap V ]=[ {U}: x^{-1} {U}x\cap U ] .$$
\end{cor}

\begin{lem}
The subspace topology defined on $G\quot H$ is locally compact.
\end{lem}
\begin{proof}
Lemma~\ref{lem:index_identity} implies that $H$ and $\tilde{H}$ act on $G/H$ by permuting cosets in the  blocks $H(gH)$ ($g\in G$) so that, making a natural identification,
\begin{equation}
\label{eq:profinite}
\tilde{H}\leq \prod_{H(gH)\subset G/H} \mathrm{Sym}(H(gH)).
\end{equation}
It may be checked that $\tilde{H}$ is a closed subgroup under this identification and that the topology on $\tilde{H}$ is equivalent to the subspace topology for the product topology on $\prod_{H(gH)\subset G/H} \mathrm{Sym}(H(gH))$. Since commensurability of $H$ implies that
each block $H(gH)$ is finite, it follows by Tychonov's theorem (see for example \cite{\Patty} Theorem 6.50) that $\tilde{H}$ is compact, and is thus a compact, open neighbourhood of $\ident$.\end{proof}

The subgroup $\langle a\rangle$ is commensurated by the group $\mathrm{BS}(m,n)$ for $m,n\neq 0$ and the kernel of the map $\pi:\mathrm{BS}(m,n)\rightarrow \mathrm{Sym}(\mathrm{BS}(m,n)/\langle a\rangle)$ is a subgroup of $\langle a\rangle$.
Then $\ker \pi$ is trivial if $|m|\neq |n|$, by Lemma \ref{lem:BSnormal}, and $\mathrm{BS}(m,n)$ embeds as a dense subgroup of the 
 totally disconnected, locally compact group $\mathrm{BS}(m,n)\quot \langle a \rangle$. This group is denoted by $G_{m,n}$ in the sequel. 
 
If $|m|=|n|$, we have $a^i(gH)=gH$ for all $g\in \mathrm{BS}(m, \pm m)$ if and only if $m$ divides $i$ and so 
 $\ker(\pi)=\langle a^m\rangle$.
Factoring $\mathrm{BS}(m,\pm m)$ by $\ker(\pi)$ then makes the commensurated subgroup $\langle a\rangle/\langle a^m\rangle$ finite, whence $\tilde H$ is finite and discrete. Hence 
$$
(\mathrm{BS}(m,\pm m)/\langle a^m\rangle)\quot(\langle a\rangle/\langle a^m\rangle)\cong \mathrm{BS}(m,\pm m)/\langle a^m\rangle
$$ (which we will denote as $G_{m,\pm m}$) and is discrete. Since the construction yields nothing new in this case, the focus of the rest of the article is on the case $|m|\neq |n|$.
 

\subsection{An alternative construction of $G_{m,n}$}

Gal and Januszkiewicz  \cite{\GalJanHNN} embed $\mathrm{BS}(m,n)$ into a topological group as follows.
Recall 
 the {\em Bass-Serre tree} for a graph of groups \cite{\Bass, \Serre}. In the case of  
 $\mathrm{BS}(m,n)$ the tree has a vertex for each  coset of $\langle a \rangle$, and $(u\langle a\rangle,v\langle a\rangle)$ is a directed edge labeled $t^{\pm 1}$ if $v\langle a\rangle=ua^it^{\pm 1}\langle a\rangle$. As an example, part of the Bass-Serre tree for the group  $\mathrm{BS}(2,3)$ is shown in Figure \ref{fig:BassSerre}. Note that each vertex will have degree 5 in this example (two outgoing edges labeled $t$ and three labeled $t^{-1}$).

\begin{center}
\begin{figure}[htb]
\begin{tikzpicture}[scale=.77, ->,>=stealth',shorten >=1pt,auto,node distance=2.5cm,
scale=1.3]
\tikzstyle{every state}=[fill=white,draw=black,text=black]

         \node  [state] (a) at (5,5) {$\epsilon$};
   \node  [state] (ta) at (7,8) {$t$};
   \node  [state] (ata) at (3,8) {$at$};

   \node  [state] (Ta) at (8,2) {$t^{-1}$};
   \node  [state] (aTa) at (5,2) {$at^{-1}$};
      \node  [state] (aaTa) at (2,2) {$a^2t^{-1}$};
   
 \path (a) edge [left] node {$t$} (ta);   
    \path (a) edge [left] node {$t$} (ata);   

    \path (Ta) edge [left] node {$t$} (a);   
    \path (aTa) edge [left] node {$t$} (a);   
       \path (aaTa) edge [left] node {$t$} (a);   

     \node  [state] (taTa) at (7,5) {$tat^{-1}$};
   \node  [state] (taaTa) at (9,5) {$ta^2t^{-1}$};
         \path (taaTa) edge [left] node {$t$} (ta);   
           \path (taTa) edge [left] node {$t$} (ta);

     \node  [state] (tata) at (2,11) {$tat$};
   \node  [state] (atata) at (4,11) {$atat$};
         \path (ata) edge [left] node {$t$} (tata);   
           \path (ata) edge [left] node {$t$} (atata);

     \node  [state] (tta) at (6,11) {$t^2$};
   \node  [state] (atta) at (8,11) {$at^2$};
         \path (ta) edge [left] node {$t$} (tta);   
           \path (ta) edge [left] node {$t$} (atta);

     \node  [state] (ataTa) at (3,5) {$atat^{-1}$};
   \node  [state] (ataaTa) at (1,5) {$ata^2t^{-1}$};
         \path (ataaTa) edge [left] node {$t$} (ata);   
           \path (ataTa) edge [left] node {$t$} (ata);

   \end{tikzpicture}
 \caption{Part of the Bass-Serre tree for $\mathrm{BS}(2,3)$. Edges labeled $t^{-1}$ correspond to traveling in reverse direction along edges labeled $t$.}

   \label{fig:BassSerre}
\end{figure}
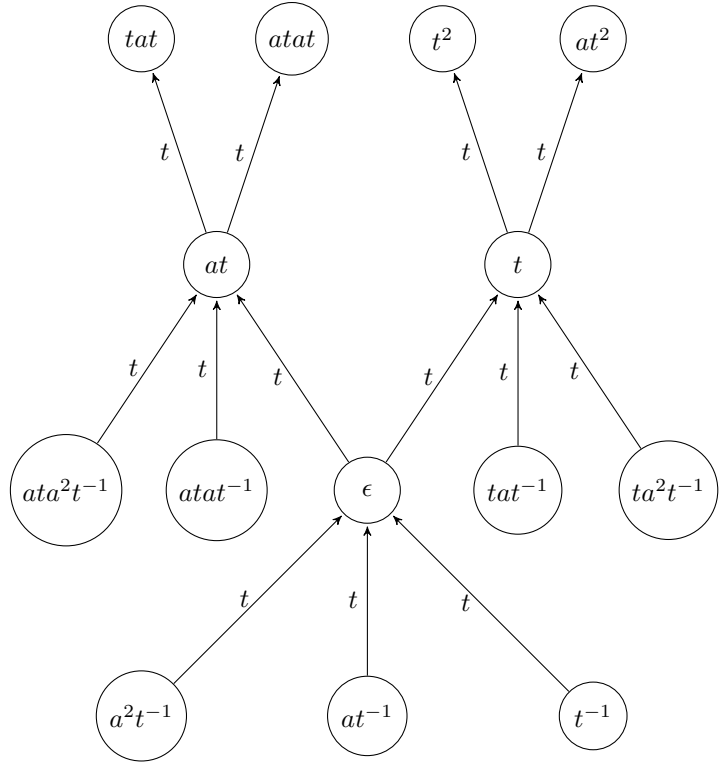
\end{center}

Let $\mathrm{Aut}(T_{m,n})$ be the  group of automorphisms of the Bass-Serre tree for $\mathrm{BS}(m,n)$ with vertex set $V(T_{m,n})$, and let   $\{N(x,\mathcal F) \ | \ x\in \mathrm{Aut}(T_{m,n}), F \subset^{\mathrm{finite}} V(T_{m,n})\}$ be a base of neighborhoods for a topology on $\mathrm{Aut}(m,n)$, where 
$$N(x, F)=\{\beta\in\mathrm{Aut}(X) \ | \ x.v=\beta.v \ \mathrm{for}  \ \mathrm{all} \ v\in F\}.$$ Since vertices correspond to (left) cosets of $\langle a\rangle$ in $\mathrm{BS}(m,n)$ then $\mathrm{Aut}(T_{m,n})\subseteq \mathrm{Sym}(\mathrm{BS}(m,n)/\langle a\rangle)$. 

To see that the closure of $\mathrm{BS}(m,n)$ in  $\mathrm{Aut}(T_{m,n})$ coincides with $G_{m,n}$ defined above, if $x\in G_{m,n}$ then either $x\in\mathrm{BS}(m,n)$ or every open set  in $\mathrm{Sym}(\mathrm{BS}(m,n)/\langle a\rangle)$ containing $x$ contains an element of $\mathrm{BS}(m,n)$. So for every finite set of cosets (vertices of $T_{m,n}$) $x$ agrees with some element of $\mathrm{BS}(m,n)$ and so preserves adjacencies of the Bass-Serre tree.







\section{Computing scales for $G_{m,n}$}\label{sec:compute}

In this section the scales of elements of $G_{m,n}$ are computed when $|m|\neq |n|$. For convenience we will abuse notation  and 
identify elements and subsets of 
 $\mathrm{BS}(m,n)$  with their images under the embedding $\pi$.

Since $\mathrm{BS}(m,n)$ is dense in $G_{m,n}$, 
 Corollary \ref{cor:dense} shows that in order to compute scales  in $G_{m,n}$, it
suffices to compute the scale of elements in $\mathrm{BS}(m,n)$. If $V$ is a compact open subgroup of $G_{m,n}$ then  Corollary
\ref{cor:indexG} shows that the index  $[V: x^{-1}Vx\cap V]$ is the same as the index  $[U: x^{-1}Ux\cap U]$ where $U$ is the intersection of $V$ with $\mathrm{BS}(m,n)$. It follows that scales in  $G_{m,n}$ can be computed by working entirely in $\mathrm{BS}(m,n)$.

We start by considering the case when one of $m,n$ is a proper divisor of the other.  Note that this includes the case that one of $m,n$ is $\pm 1$, in which case the scale could be computed more easily by working directly with the topological matrix group in Equation~(\ref{eq:matrix_group2}), or with its dense subgroup described in Equation~(\ref{eq:matrix_group}).

Suppose that $n=mr$ with $|r|>1$. Then 
\[\begin{array}{llll}
t^{-1}\langle a^m\rangle t\cap \langle a^m\rangle&=&\langle a^m\rangle,\\
t^{-1}\langle a^{mr^j}\rangle t\cap \langle a^m\rangle&=&\langle a^{mr^{j-1}}\rangle, & \mathrm{and}\\
t\langle a^{mr^i}\rangle t^{-1}\cap \langle a^m\rangle&=&\langle a^{mr^{i+1}}\rangle\end{array}\]
for any $i\geq 0,j>0$. 
These facts may encoded in a graph.
Define  $\Lambda$ to be the   labeled directed graph having nodes  $N=\{mr^i \ | \ i\in\N\}$,   and directed edges  \[E=\left\{(m,m), (mr^j,mr^{j-1}), (mr^i,mr^{i+1}) \ | \ j>0,i\geq 0\right\},\] where the loop $(m,m)$ and  edges $(mr^j,mr^{j-1})$ are labeled $t$, and edges $(mr^i,mr^{i+1}) $ are labeled $t^{-1}$.
Then 
 $(x,y)$ is an edge labeled $t^{\epsilon}$ if and only if $$t^{-\epsilon}\langle a^x \rangle t^{\epsilon}\cap \langle a^m\rangle = \langle a^y\rangle.$$
 A picture of part of  $\Lambda$  is shown in Figure \ref{diagramLambda}.

\begin{center}
\begin{figure}[htb]
\begin{tikzpicture}[scale=.77, ->,>=stealth',shorten >=1pt,auto,node distance=2.5cm,
scale=1.3]
\tikzstyle{every state}=[fill=white,draw=black,text=black]

         \node  [state] (mllll) at (1,2) {$m$};
   \node  [state] (lnnn) at (3,2) {$mr$};
      \node  [state] (lnnm) at (5,2) {$mr^2$};
      \node  [state] (lnmm) at (7,2) {$mr^3$};
      \node  [state] (lmmm) at (9,2) {$mr^4$};
       \node   (nllll) at (10.8,2.1) {};
  \node   (nllllx) at (10.8,1.9) {};

\path (mllll) edge [loop, out=210, in=150, distance=8mm] node [label={[label distance=-10pt] 10:$t$}]  {} (mllll);

   \path (mllll) edge [bend left, above] node {$t^{-1}$} (lnnn);
   \path (lnnn) edge [bend left, below] node {$t$} (mllll);   
   
   \path (lnnn) edge [bend left, above] node {$t^{-1}$} (lnnm);
   \path (lnnm) edge [bend left, below] node {$t$} (lnnn);   
   \path (lnnm) edge [bend left, above] node {$t^{-1}$} (lnmm);
   \path (lnmm) edge [bend left, below] node {$t$} (lnnm);         
   \path (lnmm) edge [bend left, above] node {$t^{-1}$} (lmmm);
   \path (lmmm) edge [bend left, below] node {$t$} (lnmm);              
         \path (lmmm) edge [bend left, above] node {$t^{-1}$} (nllll);
   \path (nllllx) edge [bend left, below] node {$t$} (lmmm);
   \end{tikzpicture}
 \caption{Part of the graph $\Lambda$.  An edge from $x$ to $y$ labeled $t^{\epsilon}$ means that $t^{-\epsilon}\langle a^x\rangle t^{\epsilon} \cap \langle a^m\rangle=\langle a^y\rangle$.}

   \label{diagramLambda}
\end{figure}
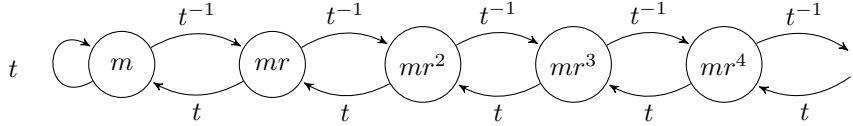
\end{center}

Recall Britton's lemma  and the notion of a pinch from Section \ref{sec:bsintro}.
\begin{lem}\label{lem:intersection}
Let $u=a^{\eta}t^{\epsilon}v$ with $\eta\in\mathbb Z$ and $\epsilon=\pm 1$  be a freely reduced word in $\mathrm{BS}(m,n)$ with no pinches. Then 
\[\begin{array}{lll}
u^{-1} \left\langle a^i\right\rangle u\cap  \left\langle a\right\rangle & = & v^{-1}\left(t^{-\epsilon} \left\langle a^i\right\rangle t^{\epsilon}\cap  \left\langle a\right\rangle\right)v\cap  \left\langle a\right\rangle.\end{array}\]
\end{lem}
\begin{proof}
 If the word $u^{-1}(a^i)^j u=v^{-1}t^{-\epsilon} a^{-\eta}(a^i)^j  a^{\eta}t^{\epsilon}v$ is in $ \left\langle a\right\rangle $
then it must contain a pinch.  There is no pinch within $t^{\epsilon}v$  (or $v^{-1}t^{-\epsilon}$) so the pinch must be $t^{-\epsilon}(a^i)^j t^{\epsilon}$, so this subword must be in $ \left\langle a\right\rangle $.\end{proof}

\begin{lem}\label{lem:LambdaGraph}
Let $w$ be a freely reduced word with no pinches, and let $p(w)$ be the word in the free monoid generated by $t$ and $t^{-1}$ obtained by removing all $a^{\pm 1}$ letters from $w$.  Consider the path in $\Lambda$ starting at the node $mr^i$, whose edge labels follow the sequence $p(w)$.
Then $$w^{-1}\langle a^{mr^i}\rangle w\cap \langle a^m\rangle=\langle a^{mr^k}\rangle$$ where $mr^k$ is the label of the node at the end of this path.
\end{lem}

\begin{proof}
If $w$ has no $t^{\pm 1}$ letters then the statement is clearly true.
Suppose for induction the statement is true for words with $q$ $t^{\pm 1}$ letters. Let $u$ be a freely reduced word with no pinches with $q+1$ $t^{\pm 1}$ letters, and write $u=a^{\eta}t^{\epsilon}v$, where $\epsilon=\pm 1$ and $\eta\in\mathbb Z$. 
Then by Lemma \ref{lem:intersection} we have
\[\begin{array}{lll}
u^{-1}\langle a^{mr^i}\rangle u\cap \langle a^m\rangle   & = & \left( v^{-1}t^{-\epsilon}\langle a^{mr^i}\rangle  t^{\epsilon}v\cap \langle a \rangle\right) \cap  \langle a^m\rangle\\
\\
 & = & \left( v^{-1} \left( t^{-\epsilon}\langle a^{mr^i}\rangle  t^{\epsilon}\cap \langle a \rangle\right)v \cap \langle a \rangle\right)\cap  \langle a^m\rangle\\
\\
 & = & \left( v^{-1} \langle a^{mr^j}\rangle v \cap \langle a \rangle\right)\cap  \langle a^m\rangle\\
   \end{array}\]
where $(mr^i,mr^j)$ is a  directed edge labeled $t^{\epsilon}$ in $\Lambda$.
The path starting at $mr^i$ labeled by $p(u)$ consists of this edge to $mr^j$, followed by a path labeled $p(v)$.
By inductive assumption $$ \left( v^{-1} \langle a^{mr^j}\rangle v \cap \langle a \rangle\right)\cap  \langle a^m\rangle= v^{-1} \langle a^{mr^j}\rangle v \cap  \langle a^m\rangle=\langle a^{mr^k}\rangle$$
since $mr^k$ is the endpoint of the path $p(v)$ starting at $mr^j$. The result follows.
\end{proof}

Consider the following example which shows how the previous lemma can be used to compute the scale. Let $w=t^{4}at^{-2}a$. Then by tracing the path $p(w)=t^4t^{-2}$ starting at $m$ through  $\Lambda$  we have
$w^{-1}\langle a^{m}\rangle w\cap \langle a^m\rangle= \langle a^{mr^2}\rangle$.
Since scale is invariant under conjugation, we could instead consider $u=t^{-2}at^4a$, in which case the path $p(u)=t^{-2}t^4$ starting at $m$ ends at $m$, so 
$u^{-1}\langle a^{m}\rangle u\cap \langle a^m\rangle= \langle a^{m}\rangle$, and we see that (the closure of) $\langle a^m\rangle$ is 
minimising for $u$, and so $s(w)=s(u)=1$.

This pre-conjugating step is the key to proving Proposition \ref{prop:divisors} below. We will need the following fact.

\begin{lem}\label{lem:conjugate}
Let $x\in \mathrm{BS}(m,n)$. Then $x$ is conjugate to a word $w\in\{a^{\pm 1},t^{\pm1}\}^*$ such that  $ww$ is freely reduced and contains no pinches. 
\end{lem}
\begin{proof}
Let $y\in\{a^{\pm 1},t^{\pm1}\}^*$ be a word  equal to the element $x$ in $ \mathrm{BS}(m,n)$, and consider the following four moves:
\begin{enumerate}\item 
if $y$ is not freely reduced, then removing a cancelling pair reduces the length of $y$, and gives a word $z$ equal to $y$ and shorter in length. 
\item if $y$ is freely reduced and $yy$ is not freely reduced, then  we must have $y=c^{\pm 1}zc^{\mp 1}$, so $z$ is conjugate  to $y$ and shorter in length. 
\item if $y$ contains a pinch, then removing gives a word $z$ equal to $y$ with fewer $t$ letters. 
\item if $y$ has no pinches and  $yy$ contains a pinch, then 
 we must have $y=a^it^{-1}vta^j$ with $i+j=km$, or  $y=a^itvt^{-1}a^j$ with $i+j=kn$. 
Conjugating by  $a^it$ [or $a^it^{-1}$ respectively] then removing the pinch we obtain a word $z=va^{nk}$ [or $z=va^{mk}$]  conjugate to $y$ with fewer $t$ letters. 
\end{enumerate}

Define an order $\prec$ on words $y,z\in  \{a^{\pm 1},t^{\pm1}\}^*$  by $z\prec y$ if $y$ has more $t$ letters than $z$, or they have the same number of $t$ letters but $y$ is longer than $z$. Then since each of the four moves produces a word that is shorter in this order, after a finite number of moves the word $w$ with the required properties is obtained.
\end{proof}


\begin{prop}\label{prop:divisors}
If $x\in  \mathrm{BS}(m,mr)$ for $|r|>1$ is equal to a word of $t$-exponent sum $\rho$, then $s(x)=1$ if $\rho\geq 0$ and $s(x)=|r|^{|\rho|}$ if $\rho<0$.
\end{prop}
\begin{proof}
Suppose $x$ has nonnegative $t$-exponent sum $\rho$ and, applying  Lemma \ref{lem:conjugate}, choose $w$  conjugate to $x$ and with $ww$ freely reduced and containing no pinches.
We consider two cases.

 If $w$ contains no $t^{-1}$ letters, then  the path starting at $m$ labeled by $p(w)=t^{\rho}$ ends at $m$, so Lemma \ref{lem:LambdaGraph} we have
$w^{-1}\langle a^{m}\rangle w\cap \langle a^m\rangle=\langle a^{m}\rangle$
and $s(w)=s(x)=1$.


Otherwise $w$ contains a $t^{-1}$ letter, and since $\rho\geq 0$ it is cyclically conjugate to a word 
 $u=t^{-1}w_1ta^{\eta}$, which is freely reduced and contains no pinches since it is a subword of $ww$.
 Suppose $u$ contains $k$ $t$ letters, and consider the word $v=t^{-k}ut^k$. Note that $k>\rho$.
This word is freely reduced and contains no pinches since $u$ starts with $t^{-1}$ and ends with $ta^{\eta}$.

Now consider the path in $\Lambda$ starting at $m$ labeled $p(v)$. It first travels $k$ steps to the right, then since $p(u)$ has exactly $k$ $t$ letters, it travels left for each $t$ and right for each $t^{-1}$ in $p(u)$ (that is, we have ensured that  it never enters the loop at $m$).
Since the subpath labeled $p(u)$ started at $mr^k$ and the $t$-exponent sum of $u$ is $\rho\geq 0$, the subpath ends at $mr^{k-\rho}$.
Finally the path $p(v)$ travels $k-\rho$ edges left and the $\rho$ edges around the loop at $m$, and so $u^{-1}\langle a^{m}\rangle u\cap \langle a^m\rangle=\langle a^{m}\rangle$
and $s(u)=s(w)=s(x)=1$.

To compute the scale of $x$ of negative $t$-exponent $\rho$, 
choose $w$  conjugate to $x$ and with $ww$ freely reduced and containing no pinches. If $w$ contains no $t$ letters then $w^{-1}$ contains no $t^{-1}$ letters, so by the argument above  the closure of 
 $\langle a^m\rangle$ is minimising for $w^{-1}$, so is minimising for $w$ by Proposition \ref{prop:scale_properties}(iii). Applying Lemma \ref{lem:LambdaGraph}, the path starting at $m$ labeled by $p(w)=t^{\rho}$ ends at $mr^{|\rho|}$ so $w^{-1}\langle a^{m}\rangle w\cap \langle a^m\rangle=\langle a^{mr^{|\rho|}}\rangle$ which has index $|r|^{|\rho|}$ in $\langle a^m\rangle $, so
 $s(w)=s(x)=|r|^{|\rho|}$.

If $w$ contains $t$ letters, then $w$ is cylically conjugate to $u=t^{-1}u_1ta^{\eta}$, and so conjugate to  $v=t^{-k}ut^k$, where $k$ is the number of $t$ letters in $u$, and $u,v$ are
 freely reduced and contain no pinches.
 
Since $v^{-1}=t^{-k}(a^{-\eta}t^{_1}u_1^{-1}t)t^k$, the argument for the second case above shows that 
 $\langle a^m\rangle$ is minimising for $v^{-1}$,  so it is minimising for $v$.
Using  Lemma \ref{lem:LambdaGraph} once again, the 
 path in $\Lambda$ starting at $m$ labeled $p(v)$ travels $k$ steps to the right, then since $p(u)$ has exactly $k$ $t$ letters, it travels left for each $t$ and right for each $t^{-1}$ in $p(u)$ (once again,  we have ensured that it never enters the loop at $m$).

Since the subpath labeled $p(u)$ started at $mr^k$ and never enters the loop, and the $t$-exponent sum of $u$ is $\rho<0$, the subpath ends at $mr^{k+|\rho|}$.
Finally the path $p(v)$ travels $k$ edges left to end at $mr^{|\rho|}$, and so$v^{-1}\langle a^{m}\rangle v\cap \langle a^m\rangle=\langle a^{mr^{|\rho|}}\rangle$ which has index $|r|^{|\rho|}$ in $\langle a^m \rangle $
so $s(v)=s(x)=|r|^{|\rho|}$.
\end{proof}

We now turn to the case where neither $m$ or $n$ is a divisor of the other. In this case we make use of M\"oller's formula (Theorem \ref{thm:Moller}), choosing $V$ to be  closure of $\langle a\rangle$. For this, it is necessary to compute $w^{-k}\langle a\rangle w^k\cap \langle a \rangle$ for arbitrary natural numbers $k$. 

\begin{cor}\label{cor:wk}
Let $x\in \mathrm{BS}(m,n)$. Then $x$ is conjugate to a word $w\in\{a^{\pm 1},t^{\pm1}\}^*$ such that  $w^k$ is freely reduced and contains no pinches for all $k\in\N$.
\end{cor}
\begin{proof}
By Lemma \ref{lem:conjugate}, $x$ is conjugate to $w$ such that  $ww$ is freely reduced and contains no pinches.
If $w^k$ contains a pair $yy^{-1}$  for some $y\in\{a^{\pm 1},t^{\pm1}\}$, then $yy^{-1}$ must lie either in $w$ or $ww$.
   If  $w^k$ contains a pinch, and the pinch is of the form $t^{\epsilon}uw^ivt^{-\epsilon}$ where $v,u$ are a prefix and suffix of $w$, then 
 $w^i$ must consist of $a^{\pm 1}$ letters, meaning $i=0$, so $ww$ contains the pinch.
\end{proof}

Define a directed graph 
 $\Omega$ as follows.  Let  $l=\lcm(m,n)$.
 The nodes of $\Omega$ are labeled by  integers from the set
 \[N=\left\{1, m\left(\frac{l}{n}\right)^i, l\left(\frac{l}{n}\right)^i, n\left(\frac{l}{m}\right)^i, l\left(\frac{l}{m}\right)^i \mid  i\in\N \right\}.\] These integers are distinct because neither $m$ nor $n$ divides the other.
 
 We define directed labeled edges in $\Omega$ in a similar way to edges in $\Lambda$: $(x,y)$ is an edge from $x$ to $y$ labeled $t^{\epsilon}$ for $\epsilon=\pm 1$ if $t^{-\epsilon}\langle a^x\rangle t^{\epsilon} \cap \langle a\rangle=\langle a^y\rangle$. Note that this time we intersect with $\langle a\rangle $ rather than  $\langle a^m\rangle $.

\begin{lem}
The edge set $E$ of $\Omega$ consists of the following edges, for all  positive integers $i,j$:
\[\begin{array}{llll}
(1,m), &(1,n), \\
\\
 (m,  m\left(\frac{l}{n}\right)), & (n, n\left(\frac{l}{m}\right)),\\
\\
 (m\left(\frac{l}{n}\right)^i,  m\left(\frac{l}{n}\right)^{i+1}), & (n\left(\frac{l}{m}\right)^i, n\left(\frac{l}{m}\right)^{i+1}),\\

\\
\mathrm{and}\\
\\
(n,m),\ & (m,n), \\
\\
(n\left(\frac{l}{m}\right)^i, l\left(\frac{l}{m}\right)^{i-1}), & (m\left(\frac{l}{n}\right)^i, l\left(\frac{l}{n}\right)^{i-1}), \\
\\
(l\left(\frac{l}{n}\right)^i, m\left(\frac{l}{n}\right)^{i+1}), & (l\left(\frac{l}{m}\right)^i, n\left(\frac{l}{m}\right)^{i+1}), \\
\\
( l\left(\frac{l}{n}\right)^{i}\left(\frac{l}{m}\right)^j,  l\left(\frac{l}{n}\right)^{i-1}\left(\frac{l}{m}\right)^{j+1}), & ( l\left(\frac{l}{n}\right)^{i}\left(\frac{l}{m}\right)^j,  l\left(\frac{l}{n}\right)^{i+1}\left(\frac{l}{m}\right)^{j-1})
\end{array}\]
where all edges in the left column are labeled $t$, and all edges in the right column are labeled $t^{-1}$.
\end{lem}
 \begin{proof} Since $1\in N$ and  \[\begin{array}{lll}
 t^{-1}\langle a\rangle t\cap \langle a\rangle & = & \langle a^m\rangle\\
 t\langle a\rangle t^{-1}\cap \langle a\rangle & = & \langle a^n\rangle\end{array}\]
 we have $m,n\in N$ and $(1,m),(1,n)\in E$ with $(1,m)$ labeled $t$ and $(1,n)$ labeled $t^{-1}$.

 Since 
 \[\begin{array}{llll}t^{-1}\langle a^m\rangle t 
 & = & \{ t^{-1}a^{|m|c} t \ : \ c\in \mathbb Z\}\\
 & = &  \{ t^{-1}a^{|m|(q\frac{l}{|m|}+r)} t \ : \ q,r\in \mathbb Z, 0\leq r<\frac{l}{|m|}\}\\
 & = &  \{ t^{-1}a^{ql}tt^{-1}a^{|m|r} t \ : \ q,r\in \mathbb Z, 0\leq r<\frac{l}{|m|}\}\\
 & = &  \{ a^{q\frac{l}{n}m}t^{-1}a^{|m|r} t \ : \ q,r\in \mathbb Z, 0\leq r<\frac{l}{|m|}\}\\
  \end{array}\]
 has intersection $\langle a^{\frac{l}{n}m}\rangle $ with $\langle a\rangle$, we have the  edge  $(m,  m\left(\frac{l}{n}\right))$ labeled $t$.
 A similar argument gives the edge $(n, n\left(\frac{l}{m}\right))$ labeled $t^{-1}$.
 
More generally for any $i\in\N$ we have
 \[\begin{array}{llll}t^{-1}\langle a^{m\left(\frac{l}{n}\right)^i}\rangle t 
 & = & \{ t^{-1}a^{|m|\left(\frac{l}{|n|}\right)^ic} t \ : \ c\in \mathbb Z\}\\
 & = &  \{ t^{-1}a^{|m|\left(\frac{l}{|n|}\right)^i(q\frac{l}{|m|}+r)} t \ : \ q,r\in \mathbb Z, 0\leq r<\frac{l}{|m|}\}\\
 & = &  \{ t^{-1}a^{ql\left(\frac{l}{|n|}\right)^i}tt^{-1}a^{|m|\left(\frac{l}{|n|}\right)^ir} t \ : \ q,r\in \mathbb Z, 0\leq r<\frac{l}{|m|}\}\\
 & = &  \{ a^{q\frac{l}{n}\left(\frac{l}{|n|}\right)^im}t^{-1}a^{|m|\left(\frac{l}{|n|}\right)^ir} t \ : \ q,r\in \mathbb Z, 0\leq r<\frac{l}{|m|}\}\\
  \end{array}\]
which  has intersection $\langle a^{\frac{l}{n}\left(\frac{l}{|n|}\right)^im}\rangle=\langle a^{\left(\frac{l}{n}\right)^{i+1}m}\rangle $ with $\langle a\rangle$, so we have  edge  $$\left(m\left(\frac{l}{n}\right)^i,  m\left(\frac{l}{n}\right)^{i+1}\right)$$ labeled $t$
for all $i\in\N$. A similar argument gives the edges $$\left(n\left(\frac{l}{m}\right)^i, n\left(\frac{l}{m}\right)^{i+1}\right)$$ labeled $t^{-1}$.

In  Figure \ref{diagram} we have drawn part of the graph $\Omega$, with these edges draw vertically down the left and right sides of the picture.

The remaining edges (which are drawn horizontally in the picture) all follow from  these facts where $c\in \N$:
\[\begin{array}{lll} t\langle a^{mc}\rangle t^{-1}&=& \langle a^{nc}\rangle,\\
t^{-1}\langle a^{lc}\rangle t&=& \langle a^{n\frac{l}{m}c}\rangle,\\
 t^{-1}\langle a^{nc}\rangle t & = &  \langle a^{mc}\rangle,\\
t^{-1}\langle a^{lc}\rangle t& = &  \langle a^{m\frac{l}{n}c}\rangle.
\end{array}\]
Since we have described an edge labeled $t$ and $t^{-1}$ starting from each node in the set $N$, there are no edges other than these.
\end{proof}

\begin{center}
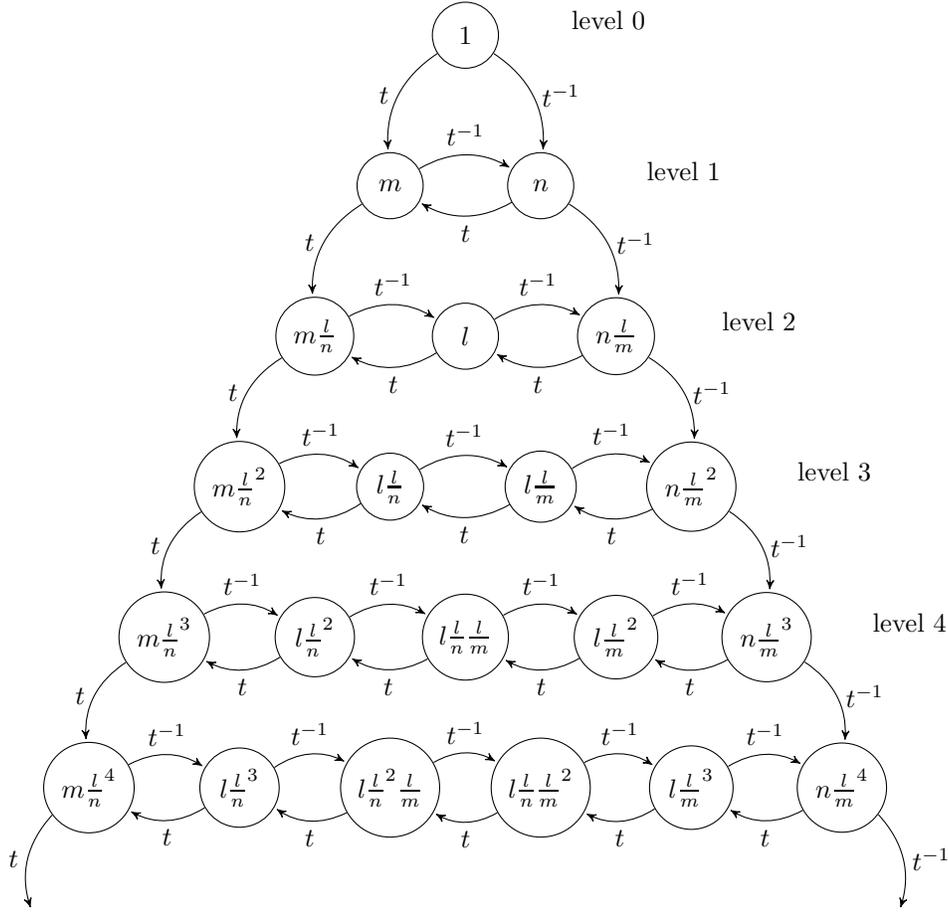
\begin{figure}[htb]
\begin{tikzpicture}[scale=.77, ->,>=stealth',shorten >=1pt,auto,node distance=2.5cm,
scale=1.3]
\tikzstyle{every state}=[fill=white,draw=black,text=black]

   \node (f) at (7.9,12.2) {level 0};
   \node (a) at (8.9,10.2) {level 1};
   \node (b) at (9.9,8.2) {level 2};
      \node (c) at (10.9,6.2)  {level 3};
      \node (d) at (11.9,4.2)  {level 4};

   \node[state]  (v) at (6,12) {$1$};

   \node[state]  (n) at (7,10) {$n$};
   \node  [state] (nl) at (8,8) {$n\frac{l}{m}$};
   \node  [state] (nll) at (9,6) {$n\frac{l}{m}^2$};
      \node  [state] (nlll) at (10,4) {$n\frac{l}{m}^3$};
         \node  [state] (nllll) at (11,2) {$n\frac{l}{m}^4$};
      \node  (nlllll) at (11.75,0.25) {};

   \node  [state] (m) at (5,10) {$m$};
      \node  [state] (ml) at (4,8) {$m\frac{l}{n}$};
   \node  [state] (mll) at (3,6) {$m\frac{l}{n}^2$};
         \node  [state] (mlll) at (2,4) {$m\frac{l}{n}^3$};
         \node  [state] (mllll) at (1,2) {$m\frac{l}{n}^4$};
         \node  (mlllll) at (0.25,0.25) {};

   \node  [state] (l) at (6,8) {$l$};
   \node  [state] (lm) at (7,6) {$l\frac{l}{m}$};
   \node  [state] (lmm) at (8,4) {$l\frac{l}{m}^2$};
      \node  [state] (lmmm) at (9,2) {$l\frac{l}{m}^3$};

   \node  [state] (ln) at (5,6) {$l\frac{l}{n}$};
   \node  [state] (lnm) at (6,4) {$l\frac{l}{n}\frac{l}{m}$};
   \node  [state] (lnmm) at (7,2) {$l\frac{l}{n}\frac{l}{m}^2$};

   \node  [state] (lnn) at (4,4) {$l\frac{l}{n}^2$};
   \node  [state] (lnnm) at (5,2) {$l\frac{l}{n}^2\frac{l}{m}$};

   \node  [state] (lnnn) at (3,2) {$l\frac{l}{n}^3$};

\path (v) edge [bend right, left] node {$t$} (m);
\path (m) edge [bend right, left] node {$t$} (ml);
\path (ml) edge [bend right, left] node {$t$} (mll);
\path (mll) edge [bend right, left] node {$t$} (mlll);
\path (mlll) edge [bend right, left] node {$t$} (mllll);
\path (mllll) edge [bend right, left] node {$t$} (mlllll);

\path (v) edge [bend left, right] node {$t^{-1}$} (n);
\path (n) edge [bend left, right] node {$t^{-1}$} (nl);
\path (nl) edge [bend left, right] node {$t^{-1}$} (nll);
\path (nll) edge [bend left, right] node {$t^{-1}$} (nlll);
\path (nlll) edge [bend left, right] node {$t^{-1}$} (nllll);
\path (nllll) edge [bend left, right] node {$t^{-1}$} (nlllll);

   \path (m) edge [bend left, above] node {$t^{-1}$} (n);
   \path (n) edge [bend left, below] node {$t$} (m);   
   
   \path (ml) edge [bend left, above] node {$t^{-1}$} (l);
   \path (l) edge [bend left, below] node {$t$} (ml);   
   
   \path (l) edge [bend left, above] node {$t^{-1}$} (nl);
   \path (nl) edge [bend left, below] node {$t$} (l);   

   \path (mll) edge [bend left, above] node {$t^{-1}$} (ln);
   \path (ln) edge [bend left, below] node {$t$} (mll);   
   
   \path (ln) edge [bend left, above] node {$t^{-1}$} (lm);
   \path (lm) edge [bend left, below] node {$t$} (ln);   
   \path (lm) edge [bend left, above] node {$t^{-1}$} (nll);
   \path (nll) edge [bend left, below] node {$t$} (lm);   

   \path (mlll) edge [bend left, above] node {$t^{-1}$} (lnn);
   \path (lnn) edge [bend left, below] node {$t$} (mlll);   
   
   \path (lnn) edge [bend left, above] node {$t^{-1}$} (lnm);
   \path (lnm) edge [bend left, below] node {$t$} (lnn);   
   \path (lnm) edge [bend left, above] node {$t^{-1}$} (lmm);
   \path (lmm) edge [bend left, below] node {$t$} (lnm);   
   \path (lmm) edge [bend left, above] node {$t^{-1}$} (nlll);
   \path (nlll) edge [bend left, below] node {$t$} (lmm);
         
   \path (mllll) edge [bend left, above] node {$t^{-1}$} (lnnn);
   \path (lnnn) edge [bend left, below] node {$t$} (mllll);   
   
   \path (lnnn) edge [bend left, above] node {$t^{-1}$} (lnnm);
   \path (lnnm) edge [bend left, below] node {$t$} (lnnn);   
   \path (lnnm) edge [bend left, above] node {$t^{-1}$} (lnmm);
   \path (lnmm) edge [bend left, below] node {$t$} (lnnm);         
   \path (lnmm) edge [bend left, above] node {$t^{-1}$} (lmmm);
   \path (lmmm) edge [bend left, below] node {$t$} (lnmm);              
         \path (lmmm) edge [bend left, above] node {$t^{-1}$} (nllll);
   \path (nllll) edge [bend left, below] node {$t$} (lmmm);
   \end{tikzpicture}
 \caption{Part of the graph $\Omega$. An edge from $x$ to $y$ labeled $t^{\epsilon}$ means that $t^{-\epsilon}\langle a^x\rangle t^{\epsilon} \cap \langle a\rangle=\langle a^y\rangle$.}
   \label{diagram}
\end{figure}
\end{center}

 \begin{cor}\label{cor:lengthDelta}
 The graph $\Omega$ is connected. The length of the shortest directed path from $m\left(\frac{l}{n}\right)^i$
to $n\left(\frac{l}{m}\right)^i$ is $i+1$ for all $i\in \N$, and also $i+1$ from $n\left(\frac{l}{m}\right)^i$
to $m\left(\frac{l}{n}\right)^i$.
\end{cor}
\begin{proof}
Each node in $N$ can be reached by a path starting at $1$ and labeled $t^it^{-j}$ for $i,j\in \N$, $i\geq j$, as follows. \begin{itemize}\item
The path $t^i$ starting at $1$ ends at $m\left(\frac{l}{n}\right)^{i-1}$, for $i\geq 1$.
\item
The path $t^it^{-j}$ starting at $1$ ends at $l\left(\frac{l}{m}\right)^{j-1}\left(\frac{l}{n}\right)^{i-j-1}$, for $i\geq 2, j\geq 1$, and $j<i$.
\item
The path $t^it^{-i}$ starting at $1$ ends at $n\left(\frac{l}{m}\right)^{i-1}$, for $i\geq 1$.
\end{itemize}

The path $t^{-i}$ starting at $m\left(\frac{l}{n}\right)^{i-1}$ ends at
 $n\left(\frac{l}{m}\right)^{i-1}$ and similarly for the reverse path $t^i$ starting at
 $n\left(\frac{l}{m}\right)^{i-1}$, so the distance from one side of $\Omega$ to the other at level $i$ is $i$.
\end{proof}

\begin{lem}\label{lem:travel}
Let $w\in\mathrm{BS}(m,n)$ be freely reduced and contain no pinches. Let $p(w)$ be the word in the free monoid generated by $t$ and $t^{-1}$
obtained from $w$ by removing 
all $a^{\pm 1}$ letters. Consider the path starting at the node $1$ in $\Omega$ and following the sequence of edges labeled by $p(w)$.
Let $r$ be the label of the node at the end of this path.
Then  $w^{-1}\langle a\rangle w\cap \langle a\rangle=\langle a^r\rangle$.
\end{lem}
\begin{proof}
The result follows from Lemma \ref{lem:intersection}  and  induction on the number of $t^{\pm 1}$ letters in $w$. If $w=a^{\eta_0}ta^{\eta_1}, \eta_0,\eta_1\in\mathbb Z$ then $w^{-1}\langle a\rangle w\cap \langle a\rangle=
a^{-\eta_1}t^{-1}\langle a\rangle ta^{\eta_1}\cap \langle a\rangle=\langle a^m\rangle$ and the path from $1$ labeled $t$ ends at $m$, with a similar result for $w=a^{\eta_0}t^{-1}a^{\eta_1}.$ 

If $w=uta^{\eta}, \eta\in \mathbb Z$, say the path $p(u)$ ends at the node $r$. Then $w^{-1}\langle a\rangle w\cap \langle a\rangle=a^{-\eta}t^{-1}u^{-1}\langle a\rangle uta^{\eta}\cap \langle a\rangle=
a^{-\eta}t^{-1}\langle a^r\rangle ta^{\eta}\cap \langle a\rangle=\langle a^s\rangle$
where $(r,s)$ is an edge labeled $t$ in $\Omega$. A similar argument applies for $w=ut^{-1}a^{\eta}$.
\end{proof}

\begin{lem}\label{lem:newlem} 
Let $w\in\{a^{\pm 1},t^{\pm 1}\}^*$ with $t$-exponent sum $\rho(w)$, and let $t_{\max}(w)$ be the maximum $t$-exponent sum of any prefix of $w$. Write $w=w_1w_2$ where  $w_1$ has $t$-exponent sum $t_{\max}(w)$, and $w_2$ has $t$-exponent sum $\mu(w)=\rho(w)-t_{\max}(w)\leq 0$. In other words, $u_1$ is the longest prefix of $u$ that has maximum $t$-exponent sum, and the $t$-exponent sum of $u_2$ is nonpositive.  Let $r$ be the number of $t^{-1}$ letters in $w$.
Then if $R>r$, the path $p(t^{R}w)$ starting at $1$ ends at distance $|\mu(w)|$ from the left side of $\Omega$, and at level $R+t_{\max}(w)$.
\end{lem}
\begin{proof}
If $w$ has length 0 then the statement is true. Suppose for induction the statement is true for all words of length $k$ and let $w$ have length $k+1$.

If $w=ua^{\pm 1}$ then $t_{\max}(w)=t_{\max}(u), \mu(w)=\mu(u)$ and  $p(t^{R}w)=p(t^{R}u) $ so the statement is true. 

If  $w=ut^{-1}$ then $t_{\max}(w)=t_{\max}(u), \mu(w)=\mu(u)-1$. The path $p(t^Ru)$ ends at level $R+t_{\max}(w)$ distance $|\mu(u)|$ from the left side of $\Omega$ by inductive assumption.
By   Corollary \ref{cor:lengthDelta} the length of level $R+t_{\max}(u)$ is $R+t_{\max}(u)$, so since
 $r<R$, the path does not reach the right side of $\Omega$, so appending $t^{-1}$ to the path does not change level, and the distance from the left side is increased by 1, so the statement is true for $w$ since $|\mu(w)|=|\mu(u)-1|=|\mu(u)|+1$.

If $w=ut$, then we have two cases. 
If $t_{\max}(w)=t_{\max}(u)+1$, then $\mu(w)=0=\mu(u)$, so the statement is true since $p(t^Ru)$ ends on the left side of $\Omega$ and appending a $t$ edge moves it down one level along a vertical edge.

Otherwise $t_{\max}(w)=t_{\max}(u)$. In this case $\mu(u)<0$, so by inductive assumption $p(t^Ru)$ ends at level $R+t_{\max}(u)$ and positive distance $|\mu(u)|$ from the 
left side of $\Omega$, so appending a $t$ edge does not change the level, and decreases the distance from the left by 1, proving the claim.
\end{proof}

\begin{prop}\label{prop:scales}
If $x\in \mathrm{BS}(m,n)$ has $t$-exponent sum $\rho$, then $s(x)= \left(\frac{\lcm(m,n)}{|n|}\right)^{\rho}$ if $\rho\geq 0$, and $s(x)=  \left(\frac{\lcm(m,n)}{|m|}\right)^{
|\rho|}$ if $\rho<0$.
\end{prop}
\begin{proof}
We first consider the case $\rho\geq 0$. Take $w$ conjugate to $x$ such that $w^k$ is freely reduced and has no pinches, given by Corollary \ref{cor:wk}.
If $w$ has no $t^{\pm 1}$ letters then clearly the closure $\langle a\rangle$ is minimising for $w$ are we are done.
If $w$ contains no $t^{-1}$ letters, then $p(w^k)$ is a path in $\Omega$ from $1$ to $m\left(\frac{\lcm(m,n)}{n}\right)^{\rho^k-1}$ down the left side $\rho^k$ edges, where $\rho$ is the number of $t$ letters in $w$. It follows (from Lemma \ref{lem:travel}) that $w^{-k}\langle a\rangle w^k\cap \langle a\rangle=\langle a^{m\left(\frac{\lcm(m,n)}{n}\right)^{\rho^k-1}}\rangle$, so by Theorem \ref{thm:Moller}
\[\begin{array}{lll}
s(w)& = & \lim_{k\rightarrow \infty} \left[{V} : {V}\cap
w^{-k}{V}w^k\right]^{\frac1{k}}\\
& = & \lim_{k\rightarrow \infty} \left(|m|\left(\frac{\lcm(m,n)}{|n|}\right)^{\rho^k-1}\right)^{\frac1{k}}\\
& = & \left(\frac{l}{|n|}\right)^{\rho}
\end{array}\]
where $V$ is the closure of $\langle a\rangle$.

Now suppose $w$ contains both $t^{-1}$ and $t$ letters. Then $w$ is conjugate to a word $u=tvt^{-1}a^{\eta}, \eta\in\mathbb Z$. Let $N$ be the number of $t^{-1}$ letters in $u$.
We will compute the scale of the words $t^{2N}u^kt^{-2N}$ which is conjugate to $x$, is freely reduced and contains no pinches.


Let  $t_{\max}$ be the maximum $t$-exponent sum of any prefix of $u$, and  write $u=u_1u_2$ where  $u_1$ has $t$-exponent sum $t_{\max}$, and $w_2$ has $t$-exponent sum $\mu=\rho-t_{\max}\leq 0$  as in Lemma \ref{lem:newlem}.
Then  
by Lemma \ref{lem:newlem}, the path labeled $p(t^{2N}u)$ starting at 1 ends at distance $\mu$ from the left side of $\Omega$, and at level $2N+t_{\max}$.  Put $L=2N+t_{\max}$ and call the endpoint of 
 $p(t^{2N}u)$
the point $P_1$.

Now $p(t^{2N}u^2)$ travels along $p(t^{2N}u)$ to  $P_1$, then along a path labeled $p(u)$.  Since $u$ has $N$ $t^{-1}$ letters, $p(t^{2N}u^2)$ also does not reach the right side of $\Omega$.

Since $t_{\max}=\rho+|\mu|$,  the path $p(u_1)$ starting at $P$  must reach the left side and either travel along it for at least one $t$ edge (if $\rho>0$), or if $\rho=0$ it will stay on level $L$ and end on the left side.

In the case that $\rho=0$, $p(t^{2n}u^k)$ will end at the point $P_1$ for all $k\geq 1$.  Thus  $p(t^{2n}u^kt^{-2N})$ will end at the same point  in $\Omega$ for all $k\geq 1$, so
the index of $u^{-k}\langle a\rangle u^k\cap \langle a\rangle$ is constant as $k$ increases, and
by Theorem \ref{thm:Moller} the scale of $u$ is 1.

In the case that $\rho>0$ and $p(t^{2N}uu_1)$ travels down a level from the point $P_1$, it must end on the left side since $u_1$ has maximal $t$-exponent sum. Since the $t$-exponent sum of $u_2$ is nonpositive, the path $p(u_2)$ from the left side stays at this level, and ends distance $|\mu|$ from the left side once again. Let $P_2$ be the endpoint of $p(t^{2N}u^2)$. Then the geodesic path from $P_1$ to $P_2$ in $\Omega$ is $t^{\mu}t^{d}t^{-\mu}$ (since the left side is closer than the right at level $2N+t_{\max}$), where $d$ is the difference in levels, and since $u$ is also a path from $P_1$ to $P_2$ and has $t$-exponent sum $\rho$, we must have $d=\rho$.

Inductively we have that $p(t^{2N}u^k)$ ends at a point $P_k$ at level $L+\rho^{k-1}$ at distance $|\mu|$ from the left side.
For $k$ sufficiently large,  the point $P_k$  lies further than $2N$ from the right side of $\Omega$, so  $p(t^{2N}u^kt^{-2N})$ ends  at a point labeled 
\[\begin{array}{lll}
l\left(\frac{l}{n}\right)^{L+\rho^{k-1}-|\mu|-2N}\left(\frac{l}{m}\right)^{|\mu|-1+2N}
& = & l\left(\frac{l}{n}\right)^{t_{\max}+\rho^{k-1}-|\mu|}\left(\frac{l}{m}\right)^{2N+|\mu|-1}.\end{array}\]

Applying Theorem \ref{thm:Moller} we have \[\begin{array}{lll}
s(x)& = & \lim_{k\rightarrow \infty} \left[{V} : {V}\cap
(t^{2N}u^kt^{-2N})^{-1}{V}(t^{2N}u^kt^{-2N})\right]^{\frac1{k}}\\
& = & \lim_{k\rightarrow \infty} \left( l\left(\frac{l}{|n|}\right)^{t_{\max}+\rho^{k-1}-|\mu|}\left(\frac{l}{|m|}\right)^{2N+|\mu|-1}\right)^{\frac1{k}}\\
& = & \left(\frac{l}{|n|}\right)^{\rho}
\end{array}\]
where $V$ is the closure of $\langle a\rangle$.

The case that $\rho<0$ is treated with the same argument as above, exchanging the left side of $\Omega$ for the right, and considering $t^{-2N}u^kt^{2N}$.
\end{proof}

Putting the results of this section together, and noting that scales of elements of $G_{m,\pm m}$ are 1, we have

\begin{cor}\label{cor:scales}
If $x\in BS(m,n)$ for $m,n\neq 0$ has $t$-exponent sum $\rho$ then $s(x)=\left(\frac{\lcm(m,n)}{|n|}\right)^{\rho}$ if $\rho\geq 0$ and  $ \left(\frac{\lcm(m,n)}{|m|}\right)^{|\rho|}$ if $\rho<0$.
\end{cor}

and so 
\begin{thm}\label{thm:main}
 The set of scales of elements in $G_{m,n}$ for $m,n\neq 0$ is $$\displaystyle \left\{  \left(\frac{\lcm(m,n)}{|n|}\right)^{\rho},  \left(\frac{\lcm(m,n)}{|m|}\right)^{\rho} \ \mid \ \rho\in\mathbb N\right\}.$$
\end{thm}

\section{The modular function and flat rank}\label{sec:flatRank}

The
{\em modular function} of a locally compact group $G$ is the homomorphism  $\Delta:G\rightarrow \mathbb R^+$ which measures how far a left-invariant Haar integral on $G$ is from being right-invariant (see for example \cite{\HewittRoss}).

In \cite{\WillStructure,\WillFurther} the second author proves that in a totally disconnected locally compact group,  $\Delta(x)=s(x)/s(x^{-1})$.
It follows from Corollary \ref{cor:scales} that the modular function for elements of $BS(m,n)$ of $t$-exponent sum $\rho$ is $\displaystyle \left|\frac{m}{n}\right|^{\rho}$.

 \begin{lem}
 If $x\in G_{m,n}$ then there is some $w\in \mathrm{BS}(m,n)$ such that $s(x)=s(w)$ and $s(x^{-1})=s(w^{-1})$.
 \end{lem}
  \begin{proof}
Let $s(x)=c$ and $s(x^{-1})=d$. Let $\iota$ be the anti-automorphism of  $G_{m,n}$ which sends each group element to its inverse. 
Since $s$ and $\iota$ are continuous, the sets $s^{-1}(c)$ and $\iota(s^{-1}(d))$ are open, so their intersection $U$ is open and contains $x$. 
Since  $ \mathrm{BS}(m,n)$ is dense in $G_{m,n}$, $U$ contains an element $w\in \mathrm{BS}(m,n)$, so $s(w)=c$, and $s(w^{-1})=d$.
  \end{proof}

 It follows that $\displaystyle \Delta(x)= \left|\frac{m}{n}\right|^{\rho}$ for some $\rho\in\mathbb Z$ for all $x\in G_{m,n}$.
 
 \begin{cor}\label{cor:texpmap}
The $t$-exponent map extends continuously to $G_{m,n}$ and is a group homomorphism.\end{cor}
 \begin{proof}
The map $\phi:G_{m,n}\rightarrow \mathbb Z$ defined by $\displaystyle \phi(x)=\log_{m/n}\left( \Delta(x)\right)$ is a well defined group homomorphism, and $\phi\mid_{\mathrm{BS}(m,n)}$ is the $t$-exponent sum map.
 \end{proof}
 
 Now that each element of $G_{m,n}$ has a well defined notion of $t$-exponent sum, we can sharpen the main theorem to
 \begin{cor}
 The scale of $x\in G_{m,n}$ is 
 $\left(\frac{\lcm(m,n)}{c}\right)^{|\rho|}$  where $\rho$ is the $t$-exponent sum of $x$,  $c=|m|$ if $\rho\geq 0$ and $c=|n|$ if $\rho<0$.
 \end{cor}

A subgroup $P$ of a totally disconnected locally compact group is said to be {\em flat} if some compact open subgroup $V$ is minimizing for all $x\in P$. For example, if $x\in G$ is any element then $\left\langle x\right\rangle $ is flat by Proposition \ref{prop:scale_properties}$(iii)$.
Let $P$ be flat in $G$, with $V$ minimizing for each $x\in P$, 
consider the set $P_1=\{x\in P \ | \ s(x)=s(x^{-1})=1\}$. It follows from Proposition \ref{prop:scale_properties}$(v)$ that $P_1$ is a subgroup of $P$, and 
by Proposition \ref{prop:scale_properties}$(iv)$ that
$P_1$ is normal. This subgroup is 
 called the {\em uniscalar subgroup} of $P$. The second author showed that  
$P$ modulo its uniscalar subgroup is free abelian of some rank $r\in \N\cup \{\infty\}$ \cite{\WillTidy}. The number $r$ is called the {\em flat rank} of $P$. 
Define the {\em flat rank} of a totally disconnected locally compact group $G$  to be the supremum of the flat rank over all  flat subgroups of $G$.


\begin{prop}
For $m,n\neq 0$, the flat rank of $G_{m,n}$ is 1 for $|m|\neq |n|$, and 0 for $|m|=|n|$.
\end{prop}

\begin{proof}
If $|m|=|n|$, then $G_{m,n}$ has the discrete topology and all elements have scale 1, and so the flat rank of any flat subgroup is 0.

Suppose that $|m|\neq |n|$ and let $P$ be a flat subgroup of 
 $G_{m,n}$.  Let $E$ be the kernel of the $t$-exponent map defined in Corollary \ref{cor:texpmap}, which is the set of elements $x$ in $G_{m,n}$ with $s(x)=s(x^{-1})=1$.
Then $P/P_1$ embeds into  $G_{m,n}/E$.
Since the $t$-exponent sum map is surjective, $G_{m,n}/E$ is isomorphic to $\mathbb Z$.
Hence $P/P_1$ embeds into $\mathbb Z$ and the rank of $P/P_1$ is at most $1$.
 Since $\langle t\rangle$ is a flat subgroup of flat rank 1, the flat rank of  $G_{m,n}$ is equal to 1.
 \end{proof}


  \section{The local structure of $G_{m,n}$}
 \label{sec:local}

In this section we give a more detailed description of the closure of $\langle a\rangle$ in $G_{m,n}$.  Using this description and the machinery developed by the second author in \cite{\WillStructure, \WillFurther, \WillTidy}
 the computation of scale becomes much faster.
  Since we wish to reprove results from the previous two sections, we only assume results from Sections  \ref{sec:scale} to  \ref{sec:construction}. We do however continue to abuse notation  and 
identify elements and subsets of 
 $\mathrm{BS}(m,n)$  with their images under the embedding $\pi$.

Recall that an  inverse (or projective) system of groups and homomorphisms is  a family  of  groups   $(A_i)_{i\in I}$  indexed by a directed poset $(I, \leq)$, and a family of homomorphisms $f_{ij}: A_j \rightarrow  A_i$ for all $i \leq j$ 
 with the following properties: 
\begin{enumerate}
\item $f_{ii}$ is the identity on $A_i$,
\item $f_{ik} = f_{ij} \circ f_{jk}$ for all $i \leq j \leq k$.\end{enumerate}
The {\em inverse limit} of the inverse system $((A_i)_{i\in I}, (f_{ij})_{i\leq  j \in I})$ is the subgroup
\[\varprojlim_{i\in I} A_i=\left\{(a_i)_{i\in I} \in \prod_{i\in I} A_i \mid a_i=f_{ij}(a_j) \ \mathrm{for} \ \mathrm{all} \ i\leq j \ \mathrm{in} \ I\right\}\]
of the 
direct product 
$\prod_{i\in I} A_i $.

An example of an inverse limit is the additive group of  $p$-adic integers -- take $A_i=\mathbb Z/p^i\mathbb Z$ and  $f_{ij}$ the remainder map modulo $p^i$.

If  each $A_i$ is  finite, the inverse limit is called a {\em profinite group}. It is a consequence of Tychonov's theorem that profinite groups are compact. Conversely, every compact totally disconnected group is isomorphic to a profinite group,  \cite[Theorem 1.1.12]{\RibesZ}. If a topological group contains a dense cyclic subgroup, it is said to be {\em monothetic} \cite[Definition 9.2]{\HewittRoss}. The closure of $\langle a\rangle$ in $G_{m,n}$  is compact, totally disconnected and monothetic. Hence, by \cite[Theorem 25.16]{\HewittRoss}, 
\begin{equation}
\label{eq:product_descrption}
\overline{\langle a\rangle} \cong \prod_{p\in \Gamma} Z_p,
\end{equation} 
where $\Gamma$ is the set of prime numbers and $Z_p$ is isomorphic to the $p$-adic integers, $\mathbb{Z}_p$, or a (finite) quotient of this group. 

The first aim of this section is determine the groups $Z_p$ by examining the details of the construction of $G\quot H$ given in Section \ref{sec:construction} for this particular case.  As seen in  Equation~(\ref{eq:profinite}), that construction identifies $\overline{\left\langle a\right\rangle}$ with a subgroup of a product of finite permutation groups which, since $\langle a \rangle$ is cyclic, act by cyclicly permuting the cosets in $\left\langle a\right\rangle.g\left\langle a\right\rangle$ for each $g$ in $\mathrm{BS}(m,n)$. Therefore Equation \ref{eq:profinite} implies that
\[
\overline{\left\langle a\right\rangle} \leq \prod \mathbb{Z}/d\mathbb{Z},\]
where the product is indexed by the cosets $g\left\langle a\right\rangle$ in $\mathrm{BS}(m,n)/\left\langle a\right\rangle$, and $d$ is the order of the corresponding orbit $\left\langle a\right\rangle.g\left\langle a\right\rangle$.

\begin{prop}
\label{prop:local_structure}
The compact open subgroup $\overline{\left\langle a\right\rangle}$ has an open subgroup $V$ such that
$$
V \cong \prod\left\{ \mathbb{Z}_p \mid p\hbox{ is a prime divisor of }\lcm(m,n)/{m}\hbox{ or }\lcm(m,n)/{n}\right\}.
$$
The quotient $\overline{\left\langle a\right\rangle}/V$ is a finite cyclic group with order dividing $\gcd(m,n)$.
\end{prop}
\begin{proof}
Let $w\langle a\rangle\neq \langle a \rangle$ be a coset and write $w=a^{\eta_1}t^{\epsilon_1}a^{\eta_2}t^{\epsilon_2}\dots a^{\eta_k} t^{\epsilon_k}$ with $\epsilon_i=\pm 1$ and $k>0$.
To determine the size  of the $\langle a\rangle$-orbit
of $w\langle a\rangle$, consider $d$ with $a^dw\langle a\rangle =w\langle a\rangle$. It will be useful to compute $e$ such that $a^dw=wa^e$. 

The defining relation for $\mathrm{BS}(m,n)$ implies that if $a^dw=wa^e$ then 
$$
a^dw=a^{\eta_1}t^{\epsilon_1}a^{\eta_2}a^{d'}t^{\epsilon_2}\dots a^{\eta_k} t^{\epsilon_k},
$$ 
where $d'=d\frac{m}{n}$ or $d\frac{n}{m}$ depending on whether $\epsilon_1$ is equal to $1$ or $-1$. Continuing to push $d'$ past $t^{\epsilon_i}$ for $i=2$, \dots, $k$, we find that  
\begin{equation}
\label{eq:compute_e}
e = d\left(\frac{m}{n}\right)^\rho
 =d\left(\frac{\frac{\lcm(m,n)}{n}}{\frac{\lcm(m,n)}{m}}\right)^\rho,
 \end{equation}
 where $\rho$ is the $t$-exponent of $w$. Indeed, this calculation goes through provided that the $\gcd(m,n)$ and sufficiently high powers of $\frac{\lcm(m,n)}{n}$ and $\frac{\lcm(m,n)}{m}$ divide $d$. Hence there are $\langle a\rangle$-orbits with orders 
\begin{equation}
\label{eq:d}
 d = \gcd(m,n) \left(\frac{\lcm(m,n)}{m}\right)^r\left(\frac{\lcm(m,n)}{n}\right)^s
\end{equation}
 for some $r,s\in \mathbb{N}$. Since 
 $\frac{\lcm(m,n)}{n}$ and $\frac{\lcm(m,n)}{m}$ are relatively prime and the $t$-exponent of $w$ may take any integer value, it follows from (\ref{eq:compute_e}) that $r$ and $s$ may take arbitrarily high values. Therefore, if $p$ is a prime divisor of $\frac{\lcm(m,n)}{n}$ or $\frac{\lcm(m,n)}{m}$, then $Z_p$ in (\ref{eq:product_descrption}) is isomorphic to $\mathbb{Z}_p$. 
 
To see that all orbits have the orders given in (\ref{eq:d}), choose a coset representative $w$ that is freely reduced and has no pinches. Then the equation $w^{-1}a^dw=a^e$ implies that $t^{-\epsilon_1}a^d t^{\epsilon_1}$ is a pinch. Hence either $m$ or $n$ divides $d$, depending on whether $\epsilon_1$ equals $-1$ or $1$. Reduction of pinches continues for all $\epsilon_i$ provided that sufficient powers of 
 $\frac{\lcm(m,n)}{m}$ and $\frac{\lcm(m,n)}{n}$ divide $d$, thus showing that all orbits have orders as given in (\ref{eq:d}), as claimed. This implies that, if $p$ does not divide $\gcd(m,n)$, $\frac{\lcm(m,n)}{n}$ or $\frac{\lcm(m,n)}{m}$, then $Z_p$ in (\ref{eq:product_descrption}) is trivial, and that, if $p$ divides $\gcd(m,n)$ but not  $\frac{\lcm(m,n)}{n}$ or $\frac{\lcm(m,n)}{m}$, then $Z_p$ is finite cyclic with order dividing $\gcd(m,n)$.
\end{proof}

It follows from Proposition~\ref{prop:local_structure} that $G_{m,n}$ is a closed subgroup of a product of $p$-adic Lie groups with $p$ belonging to
$$
 \mathfrak{p} := \left\{ p\mid p\hbox{ is a prime divisor of }\frac{\lcm(m,n)}{n}\hbox{ or }\frac{\lcm(m,n)}{m}\right\}.
 $$ 
Hence $G_{m,n}$ belongs to the class ${\Bbb A}_{\mathfrak{p}}$ of groups defined in \cite{\HG}. It is shown in  \cite{\HG} that Hausdorff groups in the variety generated by $p$-adic Lie groups, for $p\in\mathfrak{p}$, may be approximated by groups in ${\Bbb A}_{\mathfrak{p}}$.

When the $t$-exponent of the coset representative  $w$ equals~$0$, the above calculations yield further information.
\begin{cor}
\label{cor:centralizer}
If the $t$-exponent of $w\in \mathrm{BS}(m,n)$ is equal to~$0$, then there is $d>0$ such that $w$ centralizes $\langle a^d\rangle$ and (the image under $\pi$ of) $w$ centralizes $\overline{\langle a^d\rangle}$.
\end{cor}

Using this description we now revisit Corollary \ref{cor:texpmap} where the notion of $t$-exponent sum was extended to $G_{m,n}$.
Let $x\in G_{m,n}$.

Recall that every locally compact group admits a (left-invariant) Haar measure, which is finitely additive,  invariant
under left translation by group elements, and unique up to rescaling. Let $\mu$ be such a measure on $G_{m,n}$, scaled so that 
$\mu(\overline{\langle a\rangle})=1$. Then $\mu(\overline{\langle a^j\rangle})=\frac1{j}$ for any positive integer $j$ since the measure is translation invariant (so all cosets have the same measure) and finitely additive.
The modular function
 $\Delta:G\rightarrow \mathbb R^+$ is defined as follows. For any $g\in G$ and compact set $A\subseteq G$,  $\mu(Ag)=\Delta(g)\mu(A)$.
In the case that $G$ is totally disconnected, as mentioned in Section \ref{sec:flatRank}, the second author showed that the modular function always takes rational values.
Since $\Delta$ is a continuous map from $G_{m,n}$ to $\mathbb Q^+$, and $\mathrm{BS}(m,n)$ is dense in $G_{m,n}$, there exists some  $\tilde{x}$ in $\mathrm{BS}(m,n)$ such that $\Delta(x)=\Delta(\tilde{x})$.
Then $\tilde{x}=t^{\rho}w$ where $w$ is a word of zero $t$-exponent sum.
Since $\Delta$ is a homomorphism, $\Delta(\tilde{x})=(\Delta(t))^{\rho}\Delta(w)$.

By Corollary \ref{cor:centralizer} the word $w$ centralizes a compact open subgroup, and so $\Delta(w)=1$. 
Since $\mu$ is left invariant,  \[\mu(t^{-1}\overline{\langle a^n\rangle})=\mu(\overline{\langle a^n\rangle})=\frac1{|n|}.\] Translating the compact open subgroup  $t^{-1}\overline{\langle a^n\rangle}$ by $t$ on the right we have 
\[\mu(t^{-1}\overline{\langle a^n\rangle}t)=\mu(\overline{\langle a^m\rangle})=\frac1{|m|}.\] It follows that \[\frac1{|n|}=\mu(t^{-1}\overline{\langle a^n\rangle})=\Delta(t)\mu(t^{-1}\overline{\langle a^n\rangle}t)=\Delta(t)\frac1{|m|}\] and so 
 $\Delta(t)=\left|\frac{m}{n}\right|$.

Putting this together we have $\Delta(x)=\Delta(\tilde{x})=(\Delta(t))^{\rho}\Delta(w)=\left|\frac{m}{n}\right|^{\rho}$, and so we may define the function $\phi:G_{m,n}\rightarrow \mathbb Z$ as in the proof of Corollary   \ref{cor:texpmap} which extends the  $t$-exponent sum to all of $G_{m,n}$. 

The \emph{quasi-centre}, $\hbox{QZ}(G)$, of the topological group $G$ is defined in \cite{\BurgerMozes} to be set of elements of $G$ that centralize an open subgroup of $G$. Since the intersection of any two open subgroups is an open subgroup and the set of open subgroups is invariant under automorphisms, $\hbox{QZ}(G)$ is a characteristic subgroup of $G$. It is typically not closed, but when $G$ has an open abelian subgroup it is open.

\begin{cor}
\label{cor:quasi-centre}
The quasi-centre of $G_{m,n}$ is equal to $\ker \Delta$. 
\end{cor}
\begin{proof}
Since $\overline{\langle a\rangle}$ is abelian, it is contained in $\hbox{QZ}(G_{m,n})$. The smallest subgroup of $G_{m,n}$ containing $\overline{\langle a\rangle}$ and the image under $\pi$ of all elements of  $\mathrm{BS}(m,n)$ having $t$-exponent~$0$ is  $\ker \Delta$. Hence $\hbox{QZ}(G_{m,n})$ contains this subgroup. The reverse inclusion holds because every element of the quasi-centre is unimodular. 
\end{proof}

Computation of the scale of $x\in G_{m,n}$ may now be carried out more directly with the aid of the relevant structure theory as described in \cite{\WillStructure}. Recall in particular that the compact, open subgroup, $V$ is \emph{tidy} for $x$ if, setting $V_+ = \bigcap_{k\geq0} x^{ k}Vx^{- k}$ and $V_- = \bigcap_{k\geq0} x^{-k}Vx^{k}$, we have
$$
V = V_+V_-\ \hbox{ and }\ \bigcup_{k\geq0} x^kV_+x^{-k} \hbox{ is closed}.
$$
When $V$ is tidy for $x$ we have $s(x) = [xV_+x^{-1} : V_+]$  (and so tidy subgroups are minimising).

\begin{prop}
\label{prop:scale_again}
Let $x\in G_{m,n}$ and denote the $t$-exponent of $x$ by $\rho$. Then there is a compact, open subgroup, $V$, tidy for $x$ with $V = V_+V_-$ where
\begin{eqnarray*}
V_+ &\cong& \prod\left\{ \mathbb{Z}_p \mid p\hbox{ is a prime divisor of }\lcm(m,n)/{n}\right\}\\
\hbox{ and }\  \  V_- &\cong& \prod\left\{ \mathbb{Z}_p \mid p\hbox{ is a prime divisor of }\lcm(m,n)/{m}\right\}
\end{eqnarray*}
if $\rho$ is positive and  {\it vice versa\/} if $\rho$ is negative.
The scale of $x$ is
$$
s(x) = 
\begin{cases}
\left(\frac{\lcm(m,n)}{|n|}\right)^\rho, & \hbox{ if }\rho\geq0\\
\left(\frac{\lcm(m,n)}{|m|}\right)^{-\rho}, & \hbox{ if }\rho\leq0
\end{cases}.
$$
\end{prop}
\begin{proof}
Write $x=t^{\rho}(t^{-\rho}x)= t^{\rho} w$ where $w\in G_{m,n}$ is in $\ker\Delta$. 
By Corollary~\ref{cor:quasi-centre}, there is an open subgroup $V\leq \overline{\langle a\rangle}$ that is centralised by $w$. It follows from Proposition~\ref{prop:local_structure} and \cite[Lemma~1]{\WillStructure} that we may, by passing to a subgroup if necessary, suppose that  
\begin{equation}
\label{eq:Visom}
V \cong \prod\left\{ \mathbb{Z}_p \mid p\hbox{ is a prime divisor of }\lcm(m,n)/{m}\hbox{ or }\lcm(m,n)/{n}\right\}
\end{equation}
and $V = V_+V_-$. Note that, since $V$ is an open subgroup of $\overline{\langle a \rangle}$, there is $d>0$ such that $V = \overline{\langle a^d\rangle}$ and it may also be supposed that $m$, $n$ and $(n/m)^\rho$ divide $d$.

Since $w$ centralises $V$, we have $xVx^{-1} = t^\rho Vt^{-\rho}$. Hence 
$$
 xVx^{-1} = \overline{\langle t^\rho a^d t^{-\rho}\rangle} = \overline{\langle a^e \rangle}
 \hbox{ where }e = d({\textstyle \frac{n}{m}})^\rho.
 $$ 
Then under the isomorphism (\ref{eq:Visom}) of $V$ with the product of groups of $p$-adic integers, which converts multiplicative notation to additive, conjugation by $x$ multiplies each factor $\mathbb{Z}_p$ by $({\textstyle \frac{m}{n}})^\rho$. It follows that, if $\rho$ is positive, conjugation by $x$ expands the factor $\mathbb{Z}_p$ if $p$ divides $\frac{\lcm(m,n)}{n}$ and contracts it if $p$ divides $\frac{\lcm(m,n)}{m}$, and {\it vice versa\/} if $\rho$ is negative. Therefore $V_+$ and $V_-$ are as claimed. It further follows that $\bigcap_{k\in\mathbb{Z}} x^kVx^{-k}$ is trivial whence, by \cite[Lemma~3.31(3)]{\BaWi}, $V$ is tidy for $x$. Therefore 
$s(x) = [xV_+x^{-1} : V_+]$ and has the claimed values.
\end{proof}


  \bibliography{refs} \bibliographystyle{plain}
    
      \end{document}